\documentclass[12pt,reqno]{amsart}

\usepackage{amsfonts, amsthm, amsmath}
\allowdisplaybreaks[4]

\usepackage{rotating}

\usepackage{tikz}

\usepackage{graphics}

\usepackage{amssymb}

\usepackage{amscd}

\usepackage[latin2]{inputenc}

\usepackage{t1enc}

\usepackage[mathscr]{eucal}

\usepackage{indentfirst}

\usepackage{graphicx}

\usepackage{graphics}

\usepackage{pict2e}

\usepackage{mathrsfs}

\usepackage{enumerate}
\usepackage[pagebackref]{hyperref}
\hypersetup{backref, pagebackref, colorlinks=true}
\usepackage{cite}
\usepackage{color}
\usepackage{epic}
\usepackage{hyperref} %this gives clickable references, which is nice
\numberwithin{equation}{section}
\theoremstyle{plain}

\newtheorem{theorem}{Theorem}[section]

\newtheorem{lemma}[theorem]{Lemma}

\newtheorem{corollary}[theorem]{Corollary}

\newtheorem{proposition}[theorem]{Proposition}

\theoremstyle{definition}

\newtheorem{example}[theorem]{Example}

\newtheorem{remark}[theorem]{Remark}

\newtheorem{?}[theorem]{Problem}

\def\boxit#1{\leavevmode\hbox{\vrule\vtop{\vbox{\kern.33333pt\hrule
    \kern1pt\hbox{\kern1pt\vbox{#1}\kern1pt}}\kern1pt\hrule}\vrule}}

\usepackage{collectbox}



\newcommand{\f}[1]{\ifthenelse{\equal{#1}{1}}{(q;q)_\infty}{(q^{#1};q^{#1})_{\infty}}}

\begin{document}

\title[Weighted 7-colored partitions]{On certain weighted 7-colored partitions}

\author[S. Chern]{Shane Chern}
\address[Shane Chern]{Department of Mathematics, The Pennsylvania State University, University Park, PA 16802, USA}
\email{shanechern@psu.edu; chenxiaohang92@gmail.com}

\author[D. Tang]{Dazhao Tang}

\address[Dazhao Tang]{College of Mathematics and Statistics, Chongqing University, Huxi Campus LD206, Chongqing 401331, P.R. China}
\email{dazhaotang@sina.com}

\date{\today}

\begin{abstract}
Inspired by Andrews' 2-colored generalized Frobenius partitions, we consider certain weighted 7-colored partition functions and establish some interesting Ramanujan-type identities and congruences. Moreover, we provide combinatorial interpretations of some congruences modulo 5 and 7. Finally, we study the properties of weighted 7-colored partitions weighted by the parity of certain partition statistics.
\end{abstract}

\subjclass[2010]{05A17, 11P83, 05A30}

\keywords{Weighted 7-colored partition; Ramanujan-type congruence; Unified multirank; Vector crank}

\maketitle

%\tableofcontents

%\tableofcontents

%%%%%%%%%%%%%%%%%%%%%%%%%%%%%%%%%%%%%
\section{Introduction}\label{sec1}
In his 1984 Memoir of the American Mathematical Society, Andrews \cite{And2} introduced the \emph{generalized Frobenius partition} or simply the \emph{F-partition} of $n$, which is a two-rowed array of nonnegative integers
\begin{align}\label{F-partition}
\begin{pmatrix}a_{1} &a_{2} &\cdots &a_{r}\\ b_{1} &b_{2} &\cdots &b_{r} \end{pmatrix},
\end{align}
wherein each row, which is of the same length, is arranged in non-increasing order with $n=r+\sum_{i=1}^{r}a_{i}+\sum_{i=1}^{r}b_{i}$. Furthermore, Andrews studied many general classes of F-partitions. One of them is F-partitions whose parts are taken from $k$ copies of the nonnegative integers, which is called \textit{$k$-colored F-partitions}. Let $c\phi_{k}(n)$ denote the number of $k$-colored F-partitions of $n$. Andrews derived the following generating function for $c\phi_{2}(n)$.
\begin{theorem}[Eq. (5.17), \cite{And2}]
We have
\begin{align}\label{gene fun:2-colored}
\sum_{n=0}^{\infty}c\phi_{2}(n)q^{n}=\frac{(q^{2};q^{4})_{\infty}}{(q;q^{2})_{\infty}^{4}(q^{4};q^{4})_{\infty}}.
\end{align}
\end{theorem}

Here and in the sequel, we adopt the following customary notations on partitions and $q$-series:
\begin{align*}
(a;q)_{\infty} &:=\prod_{n=0}^{\infty}(1-aq^{n}),\\
(a_{1},a_{2},\cdots,a_{n};q)_{\infty}& :=(a_{1};q)_{\infty}(a_{2};q)_{\infty}\cdots(a_{n};q)_{\infty}, \quad |q|<1.
\end{align*}

In addition, Andrews obtained the following congruence modulo 5 for $c\phi_{2}(n)$.
\begin{theorem}[Corollary 10.1, \cite{And2}]
For all $n\geq0$,
\begin{align}
c\phi_{2}(5n+3)\equiv0\pmod{5}.\label{F-partition cong}
\end{align}
\end{theorem}

Moreover, Andrews defined the \emph{D-rank} of an F-partition \eqref{F-partition} to be $a_{1}-b_{1}$ (if the parts are colored, the D-rank is the numerical magnitude of $a_{1}-b_{1}$). He then conjectured \cite[Conjecture 11.1]{And2} that the D-rank might explain the congruence \eqref{F-partition cong} combinatorially for 2-colored F-partitions. Unfortunately, this was asserted untrue by Lovejoy \cite{Lov}.

According to the generating function \eqref{gene fun:2-colored} for $c\phi_{2}(n)$ and the second author's recent work with Shishuo Fu \cite{FT} involving classical theta functions, in this paper, we study the following weighted 7-colored partitions $w_{t}(n)$ arithmetically as well as combinatorially, given by
\begin{align}
\sum_{n=0}^{\infty}w_{t}(n)q^{n}:=\frac{(q^{2};q^{2})_{\infty}}{(q;q^{2})_{\infty}^{4}(q^{t};q^{t})_{\infty}^{2}}.\label{7 colored gf}
\end{align}
In some sense, this partition function can be viewed as a generalization of Andrews' $c\phi_{2}(n)$, as one may notice that $w_{4}(n)=c\phi_{2}(n)$ for all $n\geq0$.

Similar to $c\phi_{2}(n)$, there are some Ramanujan-type congruences for $w_{t}(n)$. Below are several examples that we will prove in the later sections. When $t\equiv0\pmod{3}$, we have
\begin{align*}
w_t(3n+1)\equiv 0\pmod{4},\\
w_t(3n+2)\equiv 0\pmod{9}.
\end{align*}
Furthermore, we have
\begin{align*}
&w_t(5n+3)\equiv w_t(5n+4)\equiv 0\pmod{5},\quad \textrm{if } t\equiv 0\pmod{5},\\
&w_t(5n+4)\equiv 0\pmod{5},\quad \textrm{if } t\equiv 1\pmod{5},\\
&w_t(5n+3)\equiv 0\pmod{5},\quad \textrm{if } t\equiv 4\pmod{5}.
\end{align*}
In addition, we get
\begin{align*}
w_{2}(7n+4) &\equiv0\pmod{7},\\
w_{2}(11n+10)&\equiv 0 \pmod{11},\\
w_{3}(24n+23) &\equiv0\pmod{27}.
\end{align*}

Of course, there are more congruences beyond this list. However, these congruences can be proved by the standard $q$-series techniques and hence we do not require complicated tools like modular forms.

The rest of this paper is organized as follows. In Sect.~\ref{sec:congrunces}, we establish some Ramanujan-type identities and congruences for $w_{t}(n)$. We will then consider a unified multirank and a vector crank which can give combinatorial interpretations of certain congruences modulo 5 and 7 for $w_{t}(n)$ in Sect.~\ref{sec:com inter}. In Sect.~\ref{weighted part sta}, we study the properties of weighted 7-colored partitions weighted by the parity of certain partition statistics. Some Ramanujan-type congruences and an analog of Euler's recurrence relation are obtained. In the last section, we conclude with some remarks and questions for further study.

\section{Ramanujan-type identities and congruences}\label{sec:congrunces}
In this section, we shall prove some Ramanujan-type identities and congruences for weighted 7-colored partitions $w_{t}(n)$.

For notational convenience, we denote $\f{k}$ by $f_{k}$ for positive integers $k$ when manipulating $q$-series. Recall that Ramanujan's classical theta functions $\varphi(q)$ and $\psi(q)$ are given by
\begin{align}
\varphi(q)&:=\sum_{n=-\infty}^\infty q^{n^2}=\frac{f_{2}^5}{f_{1}^2 f_{4}^2},\\
\psi(q)&:=\sum_{n= 0}^\infty q^{n(n+1)/2}=\frac{f_{2}^2}{f_{1}}.\label{psi function}
\end{align}
It is also known that
\begin{align}
\varphi(-q)=\frac{f_{1}^2}{f_{2}}.
\end{align}

Before stating our results, we require the following $2$-dissections.
\begin{align}
f_1^2&=\frac{f_2f_8^5}{f_4^2f_{16}^2}-2q\frac{f_2f_{16}^2}{f_8},\label{eq:f12}\\
%\frac{1}{f_1^2}&=\frac{f_8^5}{f_2^5f_{16}^2}+2q\frac{f_4^2f_{16}^2}{f_2^5f_8},\label{eq:f1-2}\\
f_1^4&=\frac{f_4^{10}}{f_2^2f_8^4}-4q\frac{f_2^2f_8^4}{f_4^2},\label{eq:f14}\\
\frac{1}{f_1^4}&=\frac{f_4^{14}}{f_2^{14}f_8^4}+4q\frac{f_4^2f_8^4}{f_2^{10}}\label{eq:f1-4}.
\end{align}
These follow respectively from the $2$-dissections of $\varphi(-q)$, $\varphi^2(-q)$, and $\varphi^2(q)$ (cf. \cite[p. 40, Entry 25]{Ber1991}).

We also need the following $3$-dissections of $\psi(q)$ and $1/\varphi(-q)$,
\begin{align}
\psi(q)&=\psi(q^9)\left(\frac{1}{x(q^3)}+q\right),\label{eq:psi-dis-3}\\
\frac{1}{\varphi(-q)}&=\frac{\varphi(-q^9)^3}{\varphi(-q^3)^4}\left(1+2qx(q^3)+4q^2x(q^3)^2\right),\label{eq:1/phi-dis-3}
\end{align}
where
\begin{equation}
x(q)=\frac{(q;q^{2})_{\infty}}{(q^{3};q^{6})_{\infty}^{3}}.\label{Ram cubic frac}
\end{equation}
Here \eqref{eq:psi-dis-3} comes directly from the Jacobi's triple product identity. For \eqref{eq:1/phi-dis-3}, see \cite{BO2011}.

%By \cite[p. 40, Entry 25 (v) and (vi)]{Ber1991}, we have
%\begin{align}
%\varphi^2(q)-\varphi^2(-q) &=8q^2\psi^2(q^4),\label{theta minus func}\\
%\varphi^2(q)+\varphi^2(-q) &=2\varphi^2(q^2).\label{theta sum func}
%\end{align}

\begin{theorem}\label{Ram type fun1}
For $t\equiv0\pmod{2}$,
\begin{align}
 \sum_{n=0}^{\infty}w_{t}(2n)q^{n} &=\frac{f_{2}^{14}}{f_{1}^{9}f_{4}^{4}f_{t/2}^{2}},
 \label{1-dissection iden1}\\
 \sum_{n=0}^{\infty}w_{t}(2n+1)q^{n} &=\frac{4f_{2}^{2}f_{4}^{4}}{f_{1}^{5}f_{t/2}^{2}}.\label{1-dissection iden2}
\end{align}
\end{theorem}
\begin{proof}
%Firstly, by \eqref{theta minus func} and \eqref{theta sum func}, we have
%\begin{align*}
%\frac{f_{2}^{10}}{f_{1}^{4}f_{4}^{4}} &=\frac{f_{4}^{10}}{f_{2}^{4}f_{8}^{4}}+4q\frac{f_{8}^{4}}{f_{4}^{2}},
%\end{align*}
%which is equivalent to
%\begin{align}
%\frac{1}{f_{1}^{4}} &=\frac{f_{4}^{14}}{f_{2}^{14}f_{8}^{4}}+4q\frac{f_{4}^{2}f_{8}^{4}}{f_{2}^{10}}.\label{theta relation}
%\end{align}

From \eqref{eq:f1-4} and \eqref{7 colored gf}, we have
\begin{align}
\sum_{n=0}^{\infty}w_{t}(n)q^{n} &=\frac{f_{2}^5}{f_{1}^4f_{t}^2}\label{2-dissection}=\frac{f_{4}^{14}}{f_{2}^{9}f_{8}^{4}f_{t}^{2}}+4q\frac{f_{4}^{2}f_{8}^{4}}{f_{2}^{5}f_{t}^{2}}.
\end{align}
Extracting terms involving $q^{2n}$ and $q^{2n+1}$ in \eqref{2-dissection} and replacing $q^2$ by $q$, one easily obtains \eqref{1-dissection iden1}--\eqref{1-dissection iden2}. This completes the proof.
\end{proof}

This immediately yields
\begin{corollary}
For $t\equiv0\pmod{2}$,
\begin{align*}
w_{t}(2n+1)\equiv0\pmod{4}.
\end{align*}
\end{corollary}

\begin{remark}
We notice that the congruence $w_{4}(2n+1)=c\phi_{2}(2n+1)\equiv0\pmod{4}$ was first proved by Andrews \cite{And2}. Following the same line of proving Theorem \ref{Ram type fun1}, one may also obtain the 2-dissection of the generating function of $w_{1}(n)$.
\end{remark}

\begin{theorem}\label{Ram type fun 3-dis}
For $t\equiv0\pmod{3}$,
\begin{align}
\sum_{n=0}^\infty w_t(3n)q^n &=\frac{f_2^4 f_3^4 f_{6}}{f_{1}^8 f_{t/3}^2}\left(\frac{f_2^2 f_3^6}{f_1^2 f_6^6}+10q\frac{f_1 f_6^3}{f_2 f_3^3}\right),\\
\sum_{n=0}^\infty w_t(3n+1)q^n &=4\frac{f_2^4 f_3^4 f_{6}}{f_{1}^8 f_{t/3}^2}\left(\frac{f_2 f_3^3}{f_1 f_6^3}+q\frac{f_1^2 f_6^6}{f_2^2 f_3^6}\right),\\
\sum_{n=0}^\infty w_t(3n+2)q^n &=9\frac{f_2^4 f_3^4 f_{6}}{f_{1}^8 f_{t/3}^2}.\label{eq:wt3n+2}
\end{align}
\end{theorem}

As a consequence of Theorem \ref{Ram type fun 3-dis}, we have

\begin{corollary}\label{t mod 3 thm}
For $t\equiv0\pmod{3}$,
\begin{align*}
w_t(3n+1)\equiv 0\pmod{4},\\
w_t(3n+2)\equiv 0\pmod{9}.
\end{align*}
\end{corollary}

To prove Theorem \ref{Ram type fun 3-dis}, we need to show the following lemma.

\begin{lemma}\label{le:f25f14}
We have
\begin{align}\label{key identity}
\frac{f_{2}^5}{f_{1}^4} =\frac{f_6^4 f_9^4 f_{18}}{f_{3}^8}\Bigg(\frac{1}{x(q^3)^2}+\frac{4q}{x(q^3)}+9q^2+10q^3x(q^3)+4q^4x(q^3)^2\Bigg).
\end{align}
\end{lemma}

\begin{proof}
By Eqs.~\eqref{eq:psi-dis-3} and \eqref{eq:1/phi-dis-3}, we have
\begin{align*}
\frac{f_{2}^5}{f_{1}^4} &=\frac{\psi(q)^2}{\varphi(-q)}\\
&=\psi(q^9)^2\left(\frac{1}{x(q^3)}+q\right)^2 \frac{\varphi(-q^9)^3}{\varphi(-q^3)^4}\left(1+2qx(q^3)+4q^2x(q^3)^2\right)\\
&=\frac{f_6^4 f_9^4 f_{18}}{f_{3}^8}\Bigg(\frac{1}{x(q^3)^2}+\frac{4q}{x(q^3)}+9q^2+10q^3x(q^3)+4q^4x(q^3)^2\Bigg).
\end{align*}
This establishes \eqref{key identity}.
\end{proof}

By Lemma \ref{le:f25f14} and Eq. \eqref{Ram cubic frac}, we have the following

\begin{corollary}\label{cor:3dis}
Let
$$\sum_{n=0}^\infty a(n)q^n=\frac{f_{2}^5}{f_{1}^4}.$$
Then
\begin{align*}
\sum_{n=0}^\infty a(3n)q^n&=\frac{f_2^4 f_3^4 f_{6}}{f_{1}^8}\left(\frac{f_2^2 f_3^6}{f_1^2 f_6^6}+10q\frac{f_1 f_6^3}{f_2 f_3^3}\right),\\
\sum_{n=0}^\infty a(3n+1)q^n&=4\frac{f_2^4 f_3^4 f_{6}}{f_{1}^8}\left(\frac{f_2 f_3^3}{f_1 f_6^3}+q\frac{f_1^2 f_6^6}{f_2^2 f_3^6}\right),\\
\sum_{n=0}^\infty a(3n+2)q^n&=9\frac{f_2^4 f_3^4 f_{6}}{f_{1}^8},
\end{align*}
and hence
\begin{align*}
a(3n+1)\equiv 0\pmod{4},\\
a(3n+2)\equiv 0\pmod{9}.
\end{align*}
\end{corollary}

\begin{proof}[Proof of Theorem \ref{Ram type fun 3-dis}]
We notice that
\begin{align*}
\sum_{n=0}^\infty w_t(n)q^n=\frac{f_{2}^5}{f_{1}^4f_{t}^2}.
\end{align*}
Since $t\equiv 0\pmod{3}$, Theorem \ref{Ram type fun 3-dis} is a direct consequence of Corollary \ref{cor:3dis}.
\end{proof}

With the help of Eq.~\eqref{eq:wt3n+2}, we also have

\begin{theorem}
For all $n\geq0$,
\begin{equation}\label{eq:w3:24n+23}
w_3(24n+23)\equiv 0 \pmod{27}.
\end{equation}
\end{theorem}

\begin{proof}
Notice that
$$\sum_{n=0}^\infty w_t(3n+2)q^n=9\frac{f_2^4 f_3^4 f_{6}}{f_{1}^8 f_{t/3}^2}.$$
Hence to prove \eqref{eq:w3:24n+23}, it suffices to show
$$a_1(8n+7)\equiv 0\pmod{3},$$
where
$$\sum_{n=0}^\infty a_1(n)q^n=\frac{f_2^4 f_3^4 f_{6}}{f_{1}^{10}}.$$

With the help of \eqref{eq:f12}, it follows that, modulo $3$,
\begin{align*}
\sum_{n=0}^\infty a_1(n)q^n&=\frac{f_2^4 f_3^4 f_{6}}{f_{1}^{10}}\equiv f_1^2 f_2^7=\left(\frac{f_2f_8^5}{f_4^2f_{16}^2}-2q\frac{f_2f_{16}^2}{f_8}\right)f_2^7\\
&=\frac{f_2^8 f_8^5}{f_4^2f_{16}^2}-2q\frac{f_2^8 f_{16}^2}{f_8}.
\end{align*}

We now extract terms involving $q^{2n+1}$ and replace $q^2$ by $q$, then
\begin{align*}
\sum_{n=0}^\infty a_1(2n+1)q^n&\equiv-2\frac{f_1^8 f_{8}^2}{f_4}.
\end{align*}
Through a similar argument, we have
$$\sum_{n=0}^\infty a_1(4n+3)q^n\equiv16f_2^7 f_4^2,$$
which contains no terms of the form $q^{2n+1}$. Hence $a_1(4(2n+1)+3)=a_1(8n+7)\equiv 0\pmod{3}$.
\end{proof}
\begin{remark}
As pointed out by the referee, \eqref{eq:w3:24n+23} still holds modulo $3^6=729$, which can be proved via modular forms. However, it is unclear whether there is an elementary proof.
\end{remark}

We also have some congruences modulo $5$.

\begin{theorem}\label{Thm:mod 5}
For all $n\geq0$,
\begin{align}
&w_t(5n+3)\equiv w_t(5n+4)\equiv 0\pmod{5},\quad \emph{if } t\equiv 0\pmod{5}, \label{colored part mod 5(1)}\\
&w_t(5n+4)\equiv 0\pmod{5},\quad \emph{if } t\equiv 1\pmod{5},\label{colored part mod 5(2)}\\
&w_t(5n+3)\equiv 0\pmod{5},\quad \emph{if } t\equiv 4\pmod{5}.\label{colored part mod 5(3)}
\end{align}
\end{theorem}

\begin{proof}
If $t\equiv 0\pmod{5}$, we have
$$\sum_{n=0}^\infty w_t(n)q^n\equiv \f{1}\frac{\f{10}}{\f{5}\f{t}^2} \pmod{5}.$$
From Euler's pentagonal number theorem \cite[p. 11, Corollary 1.7]{And1}, which tells
\begin{align}\label{EPNT}
(q;q)_\infty=\sum_{n=-\infty}^\infty (-1)^n q^{n(3n+1)/2},
\end{align}
we notice that $(q;q)_\infty$ has no terms in which the power of $q$ is $3$ or $4$ mod $5$. Hence
$$w_{t}(5n+3)\equiv 0 \pmod{5},$$
and
$$w_{t}(5n+4)\equiv 0 \pmod{5}.$$

We next notice that
\begin{align*}
\sum_{n=0}^\infty w_t(n)q^n&\equiv \f{1}\f{t}^3\frac{\f{10}}{\f{5}\f{5t}} \pmod{5}.
\end{align*}
Recall that the Jacobi's identity \cite[p. 14, Theorem 1.3.9]{Ber} tells
\begin{equation}\label{JI}
(q;q)_\infty^3=\sum_{n=0}^\infty (-1)^n(2n+1)q^{n(n+1)/2}.
\end{equation}
From \eqref{EPNT} and \eqref{JI}, one readily has
$$(q;q)_\infty=E_0+E_1+E_2,$$
where $E_i$ consists of those terms in which the power of $q$ is $i$ modulo $5$, and
$$(q;q)_\infty^3=J_0+J_1+J_3,$$
where $J_i$ consists of those terms in which the power of $q$ is $i$ modulo $5$. Furthermore, we note that $J_3\equiv 0\pmod{5}$, so
$$(q;q)_\infty^3\equiv J_0+J_1 \pmod{5}.$$

If $t\equiv 1\pmod{5}$, we observe that
$$\f{t}^3\equiv J_0^*+J_1^* \pmod{5},$$
where $J_i^*$ consists of terms in which the power of $q$ is $i$ modulo $5$. It follows that,
$$\f{1}\f{t}^3\equiv (E_0+E_1+E_2)(J_0^*+J_1^*)\pmod{5},$$
which contains no terms of the form $q^{5n+4}$. Hence, $w_t(5n+4)\equiv 0\pmod{5}$.

If $t\equiv 4\pmod{5}$, we have
$$\f{t}^3\equiv J_0^*+J_4^* \pmod{5}.$$
The rest of the proof is similar.
\end{proof}

For $w_2(n)$, we have the following mod $7$ congruence.

\begin{theorem}\label{Thm:mod 7}
For all $n\geq0$,
\begin{align}
w_{2}(7n+4) &\equiv0\pmod{7}.
\end{align}
\end{theorem}

\begin{proof}
Notice that
\begin{equation*}
\sum_{n=0}^{\infty}w_{2}(n)q^{n} =\frac{\f{2}}{(q;q^{2})_{\infty}^{4}\f{2}^{2}}=\frac{\f{2}^{3}}{\f{1}^{4}}\equiv\frac{\f{1}^{3}\f{2}^{3}}{\f{7}} \pmod{7}.
\end{equation*}

Similarly, \eqref{JI} tells
$$\f{1}^3\equiv J_0+J_1+J_3\pmod{7},$$
and hence
$$\f{2}^3\equiv J_0^*+J_2^*+J_6^* \pmod{7}.$$
This time $J_i$ and $J_i^*$ consist of terms in which the power of $q$ is $i$ modulo $7$.

It follows that, modulo $7$,
$$\f{1}^3\f{2}^3\equiv (J_0+J_1+J_3)(J_0^*+J_2^*+J_6^*),$$
which contains no terms of the form $q^{7n+4}$. Hence, $w_2(7n+4)\equiv 0\pmod{7}$.
\end{proof}

%\begin{remark}
%As pointed out by the referee, with the help of modular forms, one may also obtain
%\begin{equation}
%w_{2}(11n+10)\equiv0\pmod{11}
%\end{equation}
%However, it would be interesting to find an elementary proof.
%\end{remark}

Finally, we have a mod $11$ congruence for $w_2(n)$.

\begin{theorem}\label{Thm:mod 11}
For all $n\geq0$,
\begin{align}
w_{2}(11n+10) &\equiv0\pmod{11}.\label{mod 11 cong}
\end{align}
\end{theorem}

\begin{proof}
It is easy to see that
\begin{equation*}
\sum_{n=0}^{\infty}w_{2}(n)q^{n} =\frac{\f{2}^{3}}{\f{1}^{4}}\equiv\frac{1}{\f{22}}\frac{\f{2}^{14}}{\f{1}^4} \pmod{11}.
\end{equation*}
Hence it suffices to prove that
$$a_2(11n+10)\equiv 0 \pmod{11},$$
where
$$\sum_{n=0}^\infty a_2(n)q^n = \frac{\f{2}^{14}}{\f{1}^4}.$$

According to Theorem 2 in \cite{CHL2000}, let $p=11$ and $(r,s)=(-4,14)$, we have
$$a_2(11n+120)=11^4 a_2\left(\frac{n}{11}\right),$$
where $a_2(\alpha)=0$ is $\alpha$ is not a nonnegative integer. This along with the verification of the first several $n$'s finish our proof.
\end{proof}

\section{Combinatorial interpretations}\label{sec:com inter}
\subsection{A unified multirank}
In this section, we give combinatorial interpretations of Theorems \ref{Thm:mod 5} and \ref{Thm:mod 7}. In doing so, we introduce a multirank and a vector crank for weighted 7-colored partitions. This multirank (resp.~vector crank) function enables us to divide its corresponding partition set into five (resp. seven) equivalence classes.

For a given partition $\lambda$, we let $\ell(\lambda)$ denote the number of parts in $\lambda$ and $\sigma(\lambda)$ denote the sum of all parts in $\lambda$ with the convention $\ell(\lambda)=\sigma(\lambda)=0$ for empty partition $\lambda$ of 0. Let $\mathcal{P}$ denote the set of all ordinary partitions, and $\mathcal{O}$ (resp. $\mathcal{DE}$, $\mathcal{DO}$) denote the set of all partitions into odd parts (resp. distinct even parts, distinct odd parts).

Let
\begin{align*}
\mathcal{V}_{t}=\{(\lambda_{1},\lambda_{2},\lambda_{3},\lambda_{4},\lambda_{5},t\lambda_{6},t\lambda_{7})|\lambda_{1}\in\mathcal{DE},\lambda_{2},\lambda_{3},
\lambda_{4},\lambda_{5}\in\mathcal{O},\lambda_{6},\lambda_{7}\in\mathcal{P}\}.
\end{align*}
For a vector (weighted 7-colored) partition $\overrightarrow{\lambda}\in\mathcal{V}_{t}$, we define the sum of parts function $s$, the weight function $wt$, and the multirank function $r_{7}$ by
\begin{align}
s(\overrightarrow{\lambda}) &=\sigma(\lambda_{1})+\sigma(\lambda_{2})+\sigma(\lambda_{3})+\sigma(\lambda_{4})+\sigma(\lambda_{5})+t\sigma(\lambda_{6})+t\sigma(\lambda_{7}),\\
wt(\overrightarrow{\lambda}) &=(-1)^{\ell(\lambda_{1})},\\
r_{7}(\overrightarrow{\lambda}) &=\ell(\lambda_{2})-\ell(\lambda_{3})+2(\ell(\lambda_{4})-\ell(\lambda_{5}))+2\left(\ell(\lambda_{6})-\ell(\lambda_{7})\right).\label{multirank}
\end{align}

\begin{remark}
We note that for the purpose of combinatorially obtaining Ramanujan-type congruences, there are a number of alternatives that will work equally well as multirank $r_{7}$. For example, we can use $\ell(\lambda_{2})-\ell(\lambda_{3})+2(\ell(\lambda_{4})-\ell(\lambda_{5}))+h(\ell(\lambda_{6})-\ell(\lambda_{7}))$ for $h\not\equiv0\pmod{5}$ to replace the right hand side of $r_{7}$ in \eqref{multirank}.
\end{remark}

The weighted count of vector partitions of $n$ with multirank equal to $m$, denoted by $N_{\mathcal{V}_{t}}(m,n)$, is given by
\begin{align*}
N_{\mathcal{V}_{t}}(m,n)=\sum_{\overrightarrow{\lambda}\in\mathcal{V}_{t},s(\overrightarrow{\lambda})=n\atop r_{7}(\overrightarrow{\lambda})=m}wt(\overrightarrow{\lambda}).
\end{align*}
We also define the weighted count of vector partitions of $n$ with multirank congruent to $k$ modulo $m$ by
\begin{align*}
N_{\mathcal{V}_{t}}(k,m,n)=\sum_{j=-\infty}^{\infty}N_{\mathcal{V}_{t}}(jm+k,n)=
\sum_{\overrightarrow{\lambda}\in\mathcal{V}_{t},s(\overrightarrow{\lambda})=n\atop r_{7}(\overrightarrow{\lambda})\equiv k\pmod{m}}wt(\overrightarrow{\lambda}).
\end{align*}
Thus we have the following generating function for $N_{\mathcal{V}_{t}}(m,n)$:
\begin{align}\label{gf:crank}
\sum_{m=-\infty}^{\infty}\sum_{n=0}^{\infty}N_{\mathcal{V}_{t}}(m,n)z^{m}q^{n}=\frac{(q^{2};q^{2})_{\infty}}{(zq,z^{-1}q,z^{2}q,z^{-2}q;q^{2})_{\infty}
(z^{2}q^{t},z^{-2}q^{t};q^{t})_{\infty}}.
\end{align}

According to the definition of $N_{\mathcal{V}_{t}}(m,n)$, it is nontrivial that $N_{\mathcal{V}_{t}}(m,n)$ is always nonnegative. However, the following corollary of the $q$-binomial theorem proved by Berkovich and Garvan \cite{BG} gives us an affirmative answer.
\begin{proposition}
If $|q|$, $|z|<1$, then
\begin{align*}
\frac{(az;q)_{\infty}}{(a;q)_{\infty}(z;q)_{\infty}}=\frac{1}{(a;q)_{\infty}}+\sum_{n=1}^{\infty}\frac{z^{n}}{(aq^{n};q)_{\infty}(q;q)_{n}}.
\end{align*}
\end{proposition}
This immediately yields the following corollary.
\begin{corollary}
The coefficients $N_{\mathcal{V}_{t}}(m,n)$ defined in \eqref{gf:crank} are always nonnegative.
\end{corollary}

\begin{theorem}\label{thm:mod 5}
The following relations hold for all $n\geq0$, $j\equiv1\pmod{5}$, $k\equiv4\pmod{5}$ and $l\equiv0\pmod{5}$.
\begin{align}
N_{\mathcal{V}_{j}}(0,5,5n+4)=N_{\mathcal{V}_{j}}(1,5,5n+4)=\cdots=N_{\mathcal{V}_{j}}(4,5,5n+4) &=\frac{w_{j}(5n+4)}{5},\label{eq1:mod5}\\
N_{\mathcal{V}_{k}}(0,5,5n+3)=N_{\mathcal{V}_{k}}(1,5,5n+3)=\cdots=N_{\mathcal{V}_{k}}(4,5,5n+3) &=\frac{w_{k}(5n+3)}{5},\label{eq1:mod5(2)}\\
N_{\mathcal{V}_{l}}(0,5,5n+3)=N_{\mathcal{V}_{l}}(1,5,5n+3)=\cdots=N_{\mathcal{V}_{l}}(4,5,5n+3) &=\frac{w_{l}(5n+3)}{5},\label{eq1:mod5(3)}\\
N_{\mathcal{V}_{l}}(0,5,5n+4)=N_{\mathcal{V}_{l}}(1,5,5n+4)=\cdots=N_{\mathcal{V}_{l}}(4,5,5n+4) &=\frac{w_{l}(5n+4)}{5}.\label{eq1:mod5(4)}
\end{align}
\end{theorem}

To prove Theorem \ref{thm:mod 5}, the main ingredient is the following modified Jacobi's triple product identity \cite{Gar}:
\begin{align}\label{MJTP}
\prod_{n=1}^{\infty}(1-q^{n})(1-zq^{n})(1-z^{-1}q^{n})=\sum_{n=0}^{\infty}(-1)^{n}q^{n(n+1)/2}z^{-n}\left(\frac{1-z^{2n+1}}{1-z}\right).
\end{align}
\begin{proof}[Proof of Theorem \ref{thm:mod 5}]
We will present the proof of the case $j\equiv1\pmod{5}$ to illustrate the main idea. The proofs of the remaining cases are similar.

Putting $z=\zeta_{5}=e^{2\pi i/5}$ and $t=j=5s+1$ in \eqref{gf:crank}, we see that
\begin{align*}
&\sum_{m=-\infty}^{\infty}\sum_{n=0}^{\infty}N_{\mathcal{V}_{j}}(m,n)\zeta_{5}^{m}q^{n}\\ &\quad=\sum_{n=0}^{\infty}\sum_{i=0}^{4}N_{\mathcal{V}_{j}}(i,5,n)\zeta_{5}^{i}q^{n}\\
&\quad=\frac{(q^{2};q^{2})_{\infty}}{(\zeta_{5}q,\zeta_{5}^{-1}q,\zeta_{5}^{2}q,\zeta_{5}^{-2}q;q^{2})_{\infty}(\zeta_{5}^{2}q^{5s+1},\zeta_{5}^{-2}q^{5s+1};
 q^{5s+1})_{\infty}}\\
&\quad=\frac{(q;q)_{\infty}(\zeta_{5}q^{5s+1},\zeta_{5}^{-1}q^{5s+1},q^{5s+1};q^{5s+1})_{\infty}}{(q^{5};q^{10})_{\infty}(q^{5(5s+1)};q^{5(5s+1)})_{\infty}}\\
&\quad=\frac{\sum_{m=-\infty}^{\infty}\sum_{n=0}^{\infty}(-1)^{m+n}q^{m(3m-1)/2+n(n+1)(5s+1)/2}\zeta_{5}^{-n}\left(1-\zeta_{5}^{2n+1}\right)}
 {(q^{5};q^{10})_{\infty}(q^{5(5s+1)};q^{5(5s+1)})_{\infty}(1-\zeta_{5})}.
\end{align*}
Here the last equality relies on \eqref{EPNT} and \eqref{MJTP}. Since $m(3m-1)/2\equiv0,1,2\pmod{5}$ and $n(n+1)(5s+1)/2\equiv0,1,3\pmod{5}$, it follows that ${m(3m-1)/2+n(n+1)(5s+1)/2}$ is congruent to 4 modulo 5 exactly when $m\equiv1\pmod{5}$ and $n\equiv2\pmod{5}$. This means that the coefficient of $q^{5n+4}$ in
\begin{align*}
(-1)^{m+n}q^{m(3m-1)/2+n(n+1)(5s+1)/2}\zeta_{5}^{-n}\left(1-\zeta_{5}^{2n+1}\right)
\end{align*}
is zero since $\zeta_{5}^{-n}\left(1-\zeta_{5}^{2n+1}\right)=0$ when $n\equiv2\pmod{5}$. Thus
\begin{align}
\sum_{i=0}^{4}N_{\mathcal{V}_{j}}(i,5,5n+4)\zeta_{5}^{i}=0.\label{crank eq}
\end{align}
We note that the left hand side of \eqref{crank eq} is a polynomial in $\zeta_{5}$ over $\mathbb{Z}$. It follows that $N_{\mathcal{V}_{j}}(i,5,5n+4)$ has the same value for all $0\leq i\leq4$, since the minimal polynomial for $\zeta_{5}$ over $\mathbb{Q}$ is
\begin{align*}
1+\zeta_{5}+\zeta_{5}^{2}+\zeta_{5}^{3}+\zeta_{5}^{4}.
\end{align*}
We therefore establish \eqref{eq1:mod5} for $j\equiv1\pmod{5}$. This ends our proof.
\end{proof}

\begin{table}[ht]\caption{Multiank for 7-colored partitions $w_{4}(3)$}\label{rank}
\centering
\begin{tabular}{|c|c||c|c|}
$\overrightarrow{\pi}$ &$(wt(\overrightarrow{\pi}),r_{7}(\overrightarrow{\pi}))$ &$\overrightarrow{\pi}$ &$(wt(\overrightarrow{\pi}),r_{7}(\overrightarrow{\pi}))$\\
$3_{2}$ &$(1,1)$ &$1_{2}+1_{2}+1_{5}$ &$(1,0)$\\
$3_{3}$ &$(1,-1)$ &$1_{2}+1_{3}+1_{3}$ &$(1,-1)$ \\
$3_{4}$ &$(1,2)$ &$1_{3}+1_{3}+1_{4}$ &$(1,0)$ \\
$3_{5}$ &$(1,-2)$ &$1_{3}+1_{3}+1_{5}$ &$(1,-4)$\\
$2_{1}+1_{2}$ &$(-1,1)$  &$1_{2}+1_{4}+1_{4}$ &$(1,5)$\\
$2_{1}+1_{3}$ &$(-1,-1)$ &$1_{3}+1_{4}+1_{4}$ &$(1,3)$\\
$2_{1}+1_{4}$ &$(-1,2)$  &$1_{4}+1_{4}+1_{5}$ &$(1,2)$\\
$2_{1}+1_{5}$ &$(-1,-2)$  &$1_{2}+1_{5}+1_{5}$ &$(1,-3)$\\
$1_{2}+1_{2}+1_{2}$ &$(1,3)$ &$1_{3}+1_{5}+1_{5}$ &$(1,-5)$\\
$1_{3}+1_{3}+1_{3}$ &$(1,-3)$ &$1_{4}+1_{5}+1_{5}$ &$(1,-2)$\\
$1_{4}+1_{4}+1_{4}$ &$(1,6)$ &$1_{2}+1_{3}+1_{4}$ &$(1,2)$\\
$1_{5}+1_{5}+1_{5}$ &$(1,-6)$ &$1_{2}+1_{3}+1_{5}$ &$(1,-2)$\\
$1_{2}+1_{2}+1_{3}$ &$(1,1)$ &$1_{2}+1_{4}+1_{5}$ &$(1,1)$\\
$1_{2}+1_{2}+1_{4}$ &$(1,4)$ &$1_{3}+1_{4}+1_{5}$ &$(1,-1)$\\
\end{tabular}
\end{table}

\begin{example}
In Table~\ref{rank}, we list the multirank $r_{7}$ for the total 28 weighted 7-colored partitions of 3. Note that the weighted count of multirank divides $w_{4}(3)=20$ into five residue classes, each with equal size 4.
\end{example}

\subsection{A vector crank}
For a given partition $\lambda$, the crank $c(\lambda)$ of $\lambda$ is given by \cite{AG1988}:
\begin{align*}
c(\lambda):=\begin{cases}
\lambda_{1},\, &\textrm{if}~n_{1}(\lambda)=0;\cr \mu(\lambda)-n_{1}(\lambda),\, &\textrm{if}~n_{1}(\lambda)>0,\end{cases}
\end{align*}
where $n_{1}(\lambda)$ is the number of 1's in $\lambda$, $\lambda_{1}$ is the largest part in $\lambda$, and $\mu(\lambda)$ is the number of parts larger than $n_{1}(\lambda)$. By extending the partition set $\mathcal{P}$ to a new set $\mathcal{P}^{*}$, in which two additional copies of the partition 1 (say $1^{*}$ and $1^{**}$) are added, B. Kim \cite{Kim2012, Kim2011} obtained
\begin{align*}
\frac{\f{1}}{(zq,z^{-1}q;q)_{\infty}}=\sum_{\lambda\in\mathcal{P}}wt^*(\lambda)z^{c^*(\lambda)}q^{\sigma^*(\lambda)},
\end{align*}
where $wt^*(\lambda)$, $c^*(\lambda)$ and $\sigma^*(\lambda)$ are defined as follows:
\begin{align*}
wt^*(\lambda) &:=\begin{cases}
1,\, &\textrm{if}~\lambda\in\mathcal{P},\lambda=1^*~\textrm{or}~\lambda=1^{**};\cr -1,\, &\textrm{if}~\lambda=1,\end{cases}\\
c^*(\lambda) &:=\begin{cases}
c(\lambda),\, &\textrm{if}~\lambda\in\mathcal{P};\cr 0,\, &\textrm{if}~\lambda=1;\cr1,\,&\textrm{if}~\lambda=1^*;\cr-1,\,&\textrm{if}~\lambda=1^{**},\end{cases}\\
\sigma^*(\lambda) &:=\begin{cases}
\sigma(\lambda),\, &\textrm{if}~\lambda\in\mathcal{P};\cr 1,\, &\textrm{otherwise}.\end{cases}
\end{align*}

Let
\begin{align*}
\mathcal{W}_{2}=\left\{(\lambda_{1},\lambda_{2},\lambda_{3},\lambda_{4},\lambda_{5},2\lambda_{6},2\lambda_{7})|\lambda_{1}\in\mathcal{DE},\lambda_{2},\lambda_{3},
\lambda_{4},\lambda_{5}\in\mathcal{DO},\lambda_{6},\lambda_{7}\in\mathcal{P}^{*}\right\}.
\end{align*}
For a vector partition $\overrightarrow{\lambda}\in\mathcal{W}_{2}$, we define the sum of parts function $\widetilde{s}$, the weight function $\widetilde{wt}$ and the vector crank $c_{7}$ by
\begin{align*}
\widetilde{s}(\overrightarrow{\lambda}) &=\sigma(\lambda_{1})+\sigma(\lambda_{2})+\sigma(\lambda_{3})+\sigma(\lambda_{4})+\sigma(\lambda_{5})+2\sigma^*(\lambda_{6})+2\sigma^*(\lambda_{7}),\\
\widetilde{wt}(\overrightarrow{\lambda}) &=(-1)^{\ell(\lambda_{1})}wt^*(\lambda_{6})wt^*(\lambda_{7}),\\
c_{7}(\overrightarrow{\lambda}) &=\ell(\lambda_{2})-\ell(\lambda_{3})+2\left(\ell(\lambda_{4})-\ell(\lambda_{5})\right)+c^*(\lambda_{6})+2c^*(\lambda_{7}).
\end{align*}

Let $M^*(m,n)$ denote the number of weighted 7-colored partitions in $\mathcal{W}_{2}(n)$ with vector crank equal to $m$, and $M^*(k,m,n)$ denote the number of partitions in $\mathcal{W}_{2}(n)$ with vector crank congruent to $k$ modulo $m$. We have
\begin{align}
\sum_{m=-\infty}^{\infty}\sum_{n=0}^{\infty}M^*(m,n)z^{m}q^{n}=\frac{\f{2}^{3}}{(zq,z^{-1}q,z^{2}q,z^{-2}q;q)_{\infty}}.\label{crank gf}
\end{align}

\begin{theorem}
The following relation holds for all $n\geq0$,
\begin{align*}
M^*(0,7,7n+4)=M^*(1,7,7n+4)=\cdots=M^*(6,7,7n+4)=\frac{w_{2}(7n+4)}{7}.
\end{align*}
\end{theorem}

\begin{proof}
Setting $z=\zeta_{7}=e^{2\pi i/7}$ in \eqref{crank gf}, we have
\begin{align*}
&\sum_{m=-\infty}^{\infty}\sum_{n=0}^{\infty}M^*(m,n)\zeta_{7}^{m}q^{n}\\
&\quad=\frac{\f{2}^{3}}{(\zeta_{7}q,\zeta_{7}^{-1}q,\zeta_{7}^{2}q,\zeta_{7}^{-2}q;q)_{\infty}}\\
 &\quad=\frac{\f{2}^{3}(\zeta_{7}^{3}q,\zeta_{7}^{-3}q,q;q)}{(\zeta_{7}q,\zeta_{7}^{-1}q,\zeta_{7}^{2}q,\zeta_{7}^{-2}q,\zeta_{7}^{3}q,\zeta_{7}^{-3}q,q;q)_{\infty}}\\
 &\quad=\frac{\sum_{m=0}^{\infty}\sum_{n=-\infty}^{\infty}(-1)^{m+n}(2m+1)q^{m(m+1)+n(n+1)/2}\left(1-(\zeta_{7}^{3})^{2n+1}\right)}{\f{7}(1-\zeta_{7}^3)\zeta_{7}^{3n}}.
\end{align*}
Here the last equality follows from \eqref{JI} and \eqref{MJTP} again. The rest of proof is similar to our proof of Theorem \ref{thm:mod 5}. We therefore omit the details here.
\end{proof}

\section{Partitions weighted by the parity of multirank and vector crank}\label{weighted part sta}

In \cite{CKL}, Choi, Kang, and Lovejoy studied arithmetic properties of the ordinary partition function weighted by the parity of the crank. Later, Kim \cite{Kim2015} also studied arithmetic properties of cubic partition pairs weighted by the parity of the crank analog. As an analog of their work, we consider the number of weighted 7-colored partitions $w_{t}(n)$ weighted by the parity of the multirank and vector crank, i.e.,
\begin{align*}
\sum_{n=0}^{\infty}c_{t}(n)q^{n} &:=\sum_{n=0}^{\infty}\left(\sum_{m=-\infty}^{\infty}(-1)^{m}N_{\mathcal{V}_{t}}(m,n)\right)q^{n}
 =\frac{\f{2}}{(q^{2};q^{4})_{\infty}^{2}\f{t}^{2}},\\
\sum_{n=0}^{\infty}d(n)q^{n} &:=\sum_{n=0}^{\infty}\left(\sum_{m=-\infty}^{\infty}(-1)^{m}M^{*}(m,n)\right)q^{n}=\f{2},
\end{align*}

Interestingly, $c_{t}(n)$ satisfies the following equalities and congruences.
\begin{theorem}\label{weighted analog thm }
For all $n\geq0$,
\begin{align}
&c_{t}(5n+3)= c_{t}(5n+4)= 0,\quad \emph{if } t\equiv 0\pmod{5}, \label{ana mod 5(1)}\\
&c_{t}(5n+4)\equiv 0\pmod{5},\quad \emph{if } t\equiv 1\pmod{5}, \label{ana mod 5(2)}\\
&c_{t}(5n+3)\equiv 0\pmod{5},\quad \emph{if } t\equiv 4\pmod{5}. \label{ana mod 5(3)}
\end{align}
\end{theorem}

\begin{proof}
Notice that, when $t\equiv 0\pmod{5}$,
$$\sum_{n=0}^{\infty}c_{t}(n)q^{n} =\frac{\f{2}}{(q^{2};q^{4})_{\infty}^{2}\f{t}^{2}}=\frac{\f{4}^2}{\f{2}\f{t}^2}=\frac{1}{\f{t}^2}\sum_{n=0}^\infty q^{n(n+1)},$$
where we use \eqref{psi function}. Since $n(n+1)$ is not congruent to $3$ and $4$ modulo $5$, \eqref{ana mod 5(1)} follows immediately.

When $t\not\equiv 0 \pmod{5}$, we have, modulo $5$,
\begin{align*}
\sum_{n=0}^{\infty}c_{t}(n)q^{n} =\frac{\f{2}}{(q^{2};q^{4})_{\infty}^{2}\f{t}^{2}}\equiv\frac{\f{4}^{2}\f{t}^{3}}{\f{2}\f{5t}}.
\end{align*}
It follows from \eqref{psi function} that
$$\frac{\f{2}^2}{\f{1}}=P_0+P_1+P_3,$$
where $P_i$ consists of those terms in which the power of $q$ is $i$ modulo $5$. Similar to the proofs of \eqref{colored part mod 5(2)} and \eqref{colored part mod 5(3)}, one can easily obtain \eqref{ana mod 5(2)} and \eqref{ana mod 5(3)}. This finishes the proof.
\end{proof}

\begin{remark}
We first note that \eqref{ana mod 5(1)}--\eqref{ana mod 5(3)} are interesting as it appears to be very rare that the original partition function $w_t(n)$ and its weighted versions (which are weighted by the parity of its partition statistics) satisfy the same Ramanujan-type congruences. (Here \eqref{ana mod 5(1)} can be viewed as mod $5$ congruences like \eqref{colored part mod 5(1)}.)

On the other hand, one readily sees that $c_{4}(n)=p\left(\frac{n}{2}\right)$ where $p(n)$ is the number of partitions of $n$ and $p(x)=0$ if $x$ is not an integer. From Ramanujan's congruences for $p(n)$ modulo powers of 5, it is easy to get $c_{4}(5^{\alpha}n+\lambda_{\alpha})\equiv0\pmod{5^{\alpha}}$, where $\lambda_{\alpha}$ is the least positive reciprocal of 12 modulo $5^{\alpha}$. Interestingly, this family of congruences resembles Sellers' result \cite{Sel} for $c\phi_{2}(n)$:
\begin{theorem}[Corollary 2.12, \cite{PR}]\label{Sellers thm}
For all $n\geq0$ and $\alpha\geq1$,
\begin{align*}
c\phi_{2}(5^{\alpha}n+\lambda_{\alpha})\equiv0\pmod{5^{\alpha}},
\end{align*}
where $\lambda_{\alpha}$ is the least positive reciprocal of 12 modulo $5^{\alpha}$.
\end{theorem}
\end{remark}

Relying on the combinatorial version of Euler's pentagonal number theorem \cite[p. 10, Theorem 1.6]{And1}, we obtain the following interesting corollary.
\begin{corollary}
Let $M_{e}^{*}(n)$ (resp. $M_{o}^{*}(n)$) denote the number of weighted 7-colored partitions counted by $w_{2}(n)$ with even (resp. odd) vector crank. Then
\begin{align}\label{comb inter EPNT}
d(n)=M_{e}^{*}(n)-M_{o}^{*}(n)=\begin{cases}
(-1)^{m},\, &\textrm{if}~n=m(3m\pm1);\cr 0,\, &\textrm{otherwise}.\end{cases}
\end{align}
\end{corollary}

\section{Final remarks}
At last, we collect several questions here to motivate further investigation.
\begin{enumerate}[1)]
\item Yee \cite{Yee} provided a combinatorial proof of the generating function for $k$-colored F-partitions, which was independently established by Garvan \cite{Gar2}. To the best of our knowledge, there is no combinatorial proof for $k$-colored F-partition congruences up to now. In this paper, we provide a multirank that can explain the congruence modulo 5 combinatorially for $w_{4}(n)$. However, the combinatorial correspondence between $w_{4}(n)$ and $c\phi_{2}(n)$ is unclear. New ideas are required to emerge for solving these questions.

\item As shown in Sect.~\ref{sec:com inter}, we are able to combinatorially interpret mod $5$ and $7$ congruences for $w_t(n)$ by using a unified multirank or a vector crank. However, for mod $11$ congruence \eqref{mod 11 cong}, we cannot find such an interpretation along this line and hence we cry out for a combinatorial proof.

\item  It is well known that Euler's pentagonal number theorem has a beautiful combinatorial proof, which was found by Franklin \cite{Fran1882} in 1882. Hence it is also interesting to find a Franklin-type proof of the difference between $M_{e}^{*}(n)$ and $M_{o}^{*}(n)$ in \eqref{comb inter EPNT}.

\end{enumerate}

\section*{Acknowledgement}
The authors would like to thank George E. Andrews, Shishuo Fu, Michael D. Hirschhorn and Ae Ja Yee for their helpful comments and suggestions that have improved this paper to a great extent. The authors also acknowledge the helpful suggestions made by the referee. The second author was supported by the National Natural Science Foundation of China (No.~11501061).

\end{document}